\documentclass[10pt]{amsart}
\usepackage[margin=1in,letterpaper,portrait]{geometry}
\usepackage{latexsym,amssymb,verbatim,multirow,graphicx,epsfig,enumerate,
  paralist,
  hyperref}
\usepackage{graphbox}
\usepackage{xcolor}
\usepackage[normalem]{ulem}
\usepackage{tikz}
\usepackage{bm}
\usepackage{mathabx}
\usepackage{float}

\usepackage{thm-restate}

\theoremstyle{plain}
   \newtheorem{theorem}{Theorem}[section]
   \newtheorem{proposition}[theorem]{Proposition}
   
   \newtheorem{lemma}[theorem]{Lemma}
   \newtheorem{corollary}[theorem]{Corollary}
    
   \newtheorem{conjecture}[theorem]{Conjecture}
   
   \newtheorem*{theorem*}{Theorem}
\theoremstyle{definition}
   
   \newtheorem{definition}[theorem]{Definition}
   \newtheorem{example}[theorem]{Example}

   \newtheorem{remark}[theorem]{Remark}

\numberwithin{equation}{section}

\newcommand{\op}[1]{\operatorname{#1}}

\usepackage{xparse}

\DeclareDocumentCommand \ltr { o } {%
  \IfNoValueTF {#1} {%
    \ell_W^{\mathrm{full}} %
  }{%
    \ell_{#1}^{\mathrm{full}}%
  }%
}
\DeclareDocumentCommand \lR { o } {%
  \IfNoValueTF {#1} {%
    \ell_W^{\mathrm{red}} %
  }{%
    \ell_{#1}^{\mathrm{red}}%
  }%
}
\DeclareDocumentCommand \Fred { o } {%
  \IfNoValueTF {#1} {%
    F_W^{\mathrm{red}} %
  }{%
    F_{#1}^{\mathrm{red}}%
  }%
}

\newcommand\Symm{\mathfrak{S}}

\newcommand\rank{\operatorname{rank}}
\newcommand\codim{\operatorname{codim}}

\newcommand\rk{\operatorname{rk}}

\newcommand{\RRR}{T}
\newcommand{\BBB}{\mathcal{B}}

\newcommand{\defn}[1]{{\color{blue} \it {#1}}}

\newcommand\ZZ{{\mathbb{Z}}}

\newcommand\RR{{\mathbb{R}}}

\newcommand\GL{{\mathrm{GL}}}

\newcommand{\id}{\op{id}}

\begin{document}

\title[Hurwitz orbits for parabolic quasi-Coxeter elements]{In which it is proven that, for each parabolic quasi-Coxeter element in a finite real reflection group, the orbits of the Hurwitz action on its reflection factorizations are distinguished by the two obvious invariants; and also, a lemma concerning minimal reflection generating sets}
\author{Theo Douvropoulos}
\author{Joel Brewster Lewis}
\thanks{JBL was supported in part by an ORAU Powe award and a grant from the Simons Foundation (634530)}
\maketitle

\begin{abstract}
    We prove that two reflection factorizations of a parabolic quasi-Coxeter element in a finite Coxeter group belong to the same Hurwitz orbit if and only if they generate the same subgroup and have the same multiset of conjugacy classes.  As a lemma, we classify the finite Coxeter groups for which every reflection generating set that is minimal under inclusion is also of minimum size.
\end{abstract}

\section{Introduction}

For any group $G$, there is a natural action of the $k$-stranded braid group 
\[
\BBB_k = \left\langle \sigma_1, \ldots, \sigma_{k - 1} \mid \sigma_i\sigma_j = \sigma_j\sigma_i \text{ if } |i - j| > 1 \text{ and }\sigma_i \sigma_{i + 1}\sigma_i = \sigma_{i + 1}\sigma_i\sigma_{i + 1} \text{ for } i = 1, \ldots, n - 1\right\rangle
\]
on the set of $k$-tuples of elements of $G$: the generator $\sigma_i$ acts via the \defn{(left) Hurwitz move}
\begin{equation}
\sigma_i (g_1,\cdots, \quad g_i, \quad g_{i+1}, \quad \cdots,g_k)=(g_1,\cdots, \quad g_{i+1}, \quad g^{-1}_{i+1}g_ig_{i+1}, \quad \cdots,g_k),
\label{eq:Hurwitz move}    
\end{equation}
swapping two adjacent elements and conjugating one by the other so as to preserve the product of the tuple. We call this the \defn{Hurwitz action} of $\BBB_k$ on $G^k$.  This action is of particular interest when $G$ is a reflection group and the factors $g_1, \ldots, g_k$ are reflections: in this setting, it played an important role in Hurwitz's study of branched Riemann surfaces \cite{Hurwitz} (for $G = \Symm_n$ the symmetric group) and in Bessis's proof of the $K(\pi,1)$ conjecture for complements of reflection arrangement \cite{Bessis_Annals} (for $G$ an arbitrary well generated complex reflection group).

There has also been considerable interest in fully understanding the structure of the Hurwitz action, independent of its applications to geometric problems. For abstract examples, the transitivity of the action has been interpreted \cite{Muhle_Ripoll} in terms of connectivity properties of posets associated to the group $G$ and a given generating set for it. For more explicit examples, the orbit structure for minimum length factorizations in $3$-cycles in the alternating group is completely understood \cite{Muhle_Nadeau}; the same is true for reflection factorizations in the symmetric group  \cite{Kluitmann, BIT} and in dihedral groups \cite{Sia, berger}.  For factorizations of the special \emph{Coxeter elements}, much is known in Coxeter groups \cite{bessis-dual-braid, IgusaSchiffler, BDSW, LR, Wegener_affine, WY, WY2} and complex reflection groups \cite{Ripoll, Lewis, Zach, Lazreq, Minnick}.  Of particular interest for us is the recent paper \cite{BGRW}, which gives a beautiful classification of the elements in a finite Coxeter group $W$ with the property that the Hurwitz action is transitive on their minimum-length reflection factorizations: they are precisely the \emph{parabolic quasi-Coxeter elements} of $W$ (defined below in Definition~\ref{def:pqC}).

In addition to preserving the product $g_1 \cdots g_k$ of the tuple of group elements, it is easy to see that the Hurwitz action also preserves the subgroup $H:=\langle g_1, \ldots, g_k\rangle$ generated by the factors, as well as their multiset of $H$-conjugacy classes.  In \cite[\S5]{Lewis}, it was conjectured that these invariants are sufficient to distinguish Hurwitz orbits.\footnote{In fact, the question raised in \cite{Lewis} applied to all \emph{complex} reflection groups.  However, it turns out that there are counterexamples among the minimum-length reflection factorizations in two of the exceptional complex reflection groups -- see \cite[Rem.~4.4]{LW}.}

\begin{conjecture}
\label{lewis conjecture}

Let $W$ be a finite Coxeter group and $g \in W$ an arbitrary element.  Then two reflection factorizations of $g$ belong to the same Hurwitz orbit if and only if they generate the same subgroup $H \leq W$ and have the same multiset of $H$-conjugacy classes.
\end{conjecture}

In the present paper, our main result (Theorem~\ref{thm: main}) is to show that Conjecture~\ref{lewis conjecture} is valid when the product of the factors belongs to
the class of parabolic quasi-Coxeter elements.  Along the way, we prove a lemma (Lemma~\ref{Lem: classification min=min}) that seems interesting in its own right, classifying the real reflection groups in which every reflection generating set that is \emph{minimal under inclusion} is also \emph{of minimum size}.  In particular, we show (Corollary~\ref{cor:Weyl}) that all Weyl groups have this property. 

The plan of the paper is as follows: in Section~\ref{sec:background}, we review the background on finite Coxeter groups necessary for the remainder of the paper, including the definition and properties of parabolic quasi-Coxeter elements.  In Section~\ref{sec:min}, we show (Lemma~\ref{Lem: classification min=min}) that every minimal generating set of reflections in a finite Coxeter group $W$ is of minimum size, unless $W$ contains as an irreducible factor a dihedral group of order $2m$ where $m$ is divisible by three distinct primes.  Finally, in Section~\ref{sec: main result}, we prove our main result (Theorem~\ref{thm: main}), that Conjecture~\ref{lewis conjecture} is valid whenever the element $g$ is a parabolic quasi-Coxeter element of $W$.

\section{Finite Coxeter groups and parabolic quasi-Coxeter elements}
\label{sec:background}

In this section, we provide background on finite Coxeter groups and their parabolic quasi-Coxeter elements, as necessary for the main results of the paper.  For an in-depth treatment of finite Coxeter groups, the reader may consult the classic references \cite{Humphreys}, \cite{Kane}, \cite{BjornerBrenti}, and \cite{broue_book}.

\subsection{Finite Coxeter groups}

A group $W$ is called a \defn{Coxeter group} if it is generated by a finite set $S:=\{s_1,\ldots,s_n\}$ with a presentation of the following form: for some numbers $m_{ij}$ such that $m_{ij}\in \{2,3,\ldots,\}\cup\{\infty\}$ for $ 1\leq i< j\leq n$,
\[
W=\left\langle s_1,\ldots,s_n \mid s_i^2=1,\ (s_is_j)^{m_{ij}}=1\right\rangle.
\]
We call the elements $s_i\in S$ the \defn{simple generators} of $W$ and we say that the size $n:=|S|$ is the \defn{rank} of $W$; the pair $(W,S)$ will be called a \defn{Coxeter system}. The numbers $m_{ij}$ determine the \defn{Coxeter diagram} associated with the system $(W,S)$; this is the graph with vertex set $[n]:=\{1,\ldots,n\}$ and $m_{ij} - 2$ edges between the vertices $i$ and $j$. If the Coxeter diagram of $(W,S)$ is connected, we say that $W$ is an \defn{irreducible Coxeter group}; every Coxeter group is a direct product of irreducibles.

\subsubsection*{Classification}
In this paper we only work with \defn{finite Coxeter groups}, namely Coxeter groups with finite \emph{cardinality}. These turn out to be precisely the finite subgroups of $\GL(\RR^n)$ generated by Euclidean reflections (the so-called \defn{real reflection groups}) \cite[\S6.4]{Humphreys}. We will denote by $V\cong\RR^n$ the ambient space on which they act. Coxeter classified \cite{coxeter_annals} the irreducible real reflection groups into four infinite families -- $A_n$ (the symmetric groups), $B_n$ (the hyperoctahedral groups of signed permutations), $D_n$ (index-$2$ subgroups of the hyperoctahedral groups), and $I_2(m)$ (the dihedral groups) -- and six exceptional types -- $H_3$, $H_4$, $F_4$, $E_6$, $E_7$, and $E_8$ -- where the indices in all cases correspond to ranks.

We will also be interested (see Corollary~\ref{cor:Weyl}) in a subclass of finite Coxeter groups known as \defn{crystallographic} or \defn{Weyl} groups. In terms of the numbers $m_{ij}$ in the Coxeter presentation, Weyl groups are characterized by having $m_{ij}\in\{2,3,4,6\}$ for all $i,j$. The irreducible Weyl groups are the infinite families $A_n$, $B_n$, and $D_n$, and the five exceptional types $E_6$, $E_7$, $E_8$, $F_4$, and $G_2=I_2(6)$.

\subsubsection*{Reflection factorizations} In a Coxeter group $W$, any element $t\in W$ that is conjugate to some simple generator will be called a \defn{reflection}, and $T$ will denote the set of reflections of $W$. Since $T\subset W$ is a generating set, we may consider the Cayley graph for $W$ with respect to $T$. This determines a natural length function on the elements of the group: the length of an element $g\in W$ is the number of steps in the shortest path between the identity $\id\in W$ and $g$ in the Cayley graph. Equivalently, this may be defined as the smallest number $k$ for which there exist reflections $t_1,\ldots,t_k$ such that $g=t_1\cdots t_k$. We call such factorizations \defn{reduced reflection factorizations}, and the number $\lR(g) := k$ the \defn{reflection length} of $g$. A more common construction in Coxeter groups involves the length function $\ell_S(g)$, where length is calculated in the Cayley graph with respect to the simple generators; we will \emph{not} make use of it here.

The reflection length is subadditive over products: for any $g, h$ one has $\lR(g h) \leq \lR(g) + \lR(h)$.  Thus, it determines a partial order $\leq_T$ (the \defn{absolute order}) on the elements of $W$ via
\[
u \leq_T v \qquad \iff \qquad \lR(u)+\lR(u^{-1}v)=\lR(v).
\]
In other words, $u \leq_{\RRR} v$ if and only if the reduced reflection factorizations of $u$ can be extended to give reduced reflection factorizations of $v$.

We will be particularly interested in a special type of reflection factorization that we introduce now. We say that a reflection factorization $g=t_1\cdots t_N$ is \defn{full} if the factors $t_i$ generate the \emph{full} reflection group, i.e., if $W=\langle t_1,\ldots,t_N\rangle$. Notice that every reflection factorization is trivially full in the group generated by its factors. The \defn{full reflection length} $\ltr(g)$ of $g$ is the minimum length of a full reflection factorization:
\[
\ltr(g) := \min \Big\{ k : \exists\, t_1, \ldots, t_k \in T \text{ such that } t_1 \cdots t_k = g \text{ and } \langle t_1, \ldots, t_k \rangle = W \Big\}.
\]

\subsubsection*{Parabolic subgroups}

For any subset $I\subset S$ of the simple generators $S$, the subgroup $W_I:=\langle I\rangle$ they generate will be called a \defn{standard parabolic subgroup}. Any subgroup $W'\leq W$ conjugate to some $W_I$ will be called a \defn{parabolic subgroup} of $W$. Parabolic subgroups are precisely the pointwise stabilizers $W_U$ of arbitrary subsets $U$ of the ambient space $V$ \cite[\S5.2]{Kane}, and by \cite[Lem.~1]{carter} they are generated by the reflections whose fixed hyperplanes contain $U$. It is easy with this interpretation to see that intersections of parabolic subgroups are also parabolic (in fact the collection of parabolic subgroups forms a lattice). For any element $g\in W$ we write $W_g$ for the \defn{parabolic closure} of $g$ in $W$, namely, the smallest parabolic subgroup that contains $g$. Again relying on the geometric interpretation, it is easy to see that the parabolic closure $W_g$ must equal the pointwise stabilizer $W_{V^g}$ of the fixed space $V^g$.  

The following well known lemma gives a geometric characterization of the absolute order $\leq_T$ ; we will rely on it in the proofs or our main theorems. 

\begin{lemma}[Carter's lemma]
\label{Lem: carter}
Let $W$ be a finite Coxeter group, $g$ any of its elements, and $W_g$ the parabolic closure of $g$. Then $\lR(g)=\codim(V^g)$ and for a reflection $t\in W$, the following are equivalent:
\[ 
(1) \ t\leq_T g \qquad\qquad  (2) \ V^t\supset V^g  \qquad\qquad (3) \ t\in W_g.
\]
\end{lemma}
\begin{proof}
The statement about the length $\lR(g)$ is what is usually known as Carter's lemma \cite[Lem.~2]{carter}. Carter's proof is written for Weyl groups but applies verbatim for finite Coxeter groups. The equivalence of parts $(1)$ and $(2)$ is implicit in the proof of \cite[Lem.~2]{carter}, see also \cite[Lem.~1.2.1]{bessis-dual-braid}. Parts $(2)$ and $(3)$ are equivalent by Steinberg's theorem and since $W_g=W_{V^g}$.
\end{proof}

\subsubsection*{Reflection generating sets}

Let $W$ be a Coxeter group of rank $n$.  Every Coxeter group (finite or otherwise) has a natural  geometric representation on $\RR^n$ in which each reflection fixes a hyperplane.  In this representation, any collection of fewer than $n$ reflections from $W$ have nontrivial common fixed space, and therefore cannot generate the whole group $W$.  

In the case of sets consisting of exactly $n$ reflections, there are a variety of non-obvious restrictions on which sets of reflections can generate.  For example, the following theorem of Wegener--Yahiatene shows that the conjugacy classes of the factors cannot be arbitrary.

\begin{lemma}[{\cite[Lem.~6.4]{WY2}}]
\label{Cor. refl gen sets same conj classes}
In an arbitrary Coxeter group $W$ of rank $n$, all size-$n$ reflection generating sets of $W$ determine the same multiset of conjugacy classes.
\end{lemma}

\subsection{The Hurwitz action}

As illustrated in \eqref{eq:Hurwitz move},  the Hurwitz move $\sigma_i$ applied to a tuple $(g_1, \ldots, g_k)$ moves the entry $g_{i + 1}$ one position to the left without changing its value.  The inverse \defn{(right) Hurwitz move} is given by
\begin{equation}
    \label{eq:inverse Hurwitz move}
    \sigma_i^{-1} (g_1,\cdots, \quad g_i, \quad g_{i+1}, \quad \cdots,g_k)=(g_1,\cdots, \quad  g_{i}g_{i + 1}g_{i}^{-1}, \quad g_{i}, \quad\cdots,g_k)
\end{equation}
and moves the entry $g_i$ one position to the right, also without changing its value.  By combining such moves, we can move any element of a factorization to any prescribed position (possibly conjugating other elements of the factorization), and more generally move any subsequence of entries to any prescribed set of positions unchanged as long as we do not alter their relative order in the factorization.

We will make use of the following structural lemma for the Hurwitz action on reflection factorizations throughout the paper. The proof given in \cite{LR} relies on the classification of finite Coxeter groups, but recently Wegener and Yahiatene gave a uniform proof \cite{WY2}.

\begin{lemma}[{\cite[Cor.~1.4]{LR}}]
\label{lem: LR lemma}
Let $W$ be a finite Coxeter group and $g \in W$ an element of reflection length $\lR(w)=k$. Then any factorization of $g$ into $N$ reflections (with $N\geq k$) lies in the Hurwitz orbit of some tuple $( t_1,\ldots, t_N)$ such that \[
 t_1 = t_2, \qquad\qquad
 t_3 = t_4, \qquad\qquad \cdots \qquad\qquad 
 t_{N-k-1}  = t_{N-k},
\]
and $( t_{N-k+1}, \ldots, t_N)$ is a reduced factorization of $g$.
\end{lemma}

\subsection{Parabolic quasi-Coxeter elements}

The products $s_{i_1}\cdots s_{i_n}$ of the simple generators in any order are known as Coxeter elements and they are all conjugate to each other. We will call \emph{any} element in this conjugacy class a \defn{Coxeter element} of $W$. 
The 
Coxeter elements of parabolic subgroups are similarly called (generalized) \defn{parabolic Coxeter elements} of $W$. 
Bessis showed \cite[Prop.~1.6.1]{bessis-dual-braid} that the Hurwitz action on reduced reflection factorizations of parabolic Coxeter elements is transitive. Recently Baumeister et al. \cite{BGRW} studied the Hurwitz action on reduced factorizations of \emph{arbitrary} elements in finite Coxeter groups, which led them to the following generalization of parabolic Coxeter elements. 

\begin{definition}[{Parabolic quasi-Coxeter elements}]
\label{def:pqC}
In a finite Coxeter group $W$, we call an element $g\in W$ a \defn{parabolic quasi-Coxeter element} if it admits a reduced reflection factorization $g=t_1\cdots t_k$ whose factors $t_i$ generate a parabolic subgroup of $W$. If the factors $t_i$ generate the full group $W$, we say that $g$ is a \defn{quasi-Coxeter element}.
\end{definition}

Notice that Definition~\ref{def:pqC} does not specify \emph{which} parabolic subgroup should be generated by the factors $t_i$. However, we see in the next proposition that, in fact, the subgroup is completely determined by the parabolic quasi-Coxeter element $g$.

\begin{proposition}[{\cite[Prop.~2.4.11]{Wegener_thesis} +  \cite[Thm.~1.1]{BGRW}}]
\label{prop:pqC factorization generates closure}
If $W$ is a finite Coxeter group, $g \in W$ a parabolic quasi-Coxeter element, and $g = t_1 \cdots t_k$ a reduced reflection factorization, then $\langle t_1, \ldots, t_k\rangle = W_g$.
\end{proposition}

The following proposition may be seen as a characterization for \emph{quasi-Coxeter} elements, analogous to the characterization of Coxeter elements as products of the reflections in a simple system. 

\begin{proposition}
\label{Prop: t_1...t_n=w is q-Cox}
Let $W$ be a finite Coxeter group of rank $n$, and $\{t_1,\ldots,t_n\}\subset W$ a generating set of reflections. Then the product $w=t_1\cdots t_n$ (in any order) is a quasi-Coxeter element.
\end{proposition}
\begin{proof}
The roots associated to a $n$-element generating set of reflections must be linearly independent; indeed, if they were not, the fixed spaces of the reflections would have nontrivial intersection.  Then the result is an immediate corollary of \cite[Lem.~3]{carter}, which asserts that a product of reflections in a finite Coxeter group is reduced if and only if the associated roots are linearly independent.  
\end{proof}

\subsubsection*{Characterization of parabolic quasi-Coxeter elements}
Our main result (Theorem~\ref{thm: main}) addresses Conjecture~\ref{lewis conjecture} for the class of parabolic quasi-Coxeter elements. We give below three different characterizations of parabolic quasi-Coxeter elements that we will rely on in Section~\ref{sec: main result}.

\begin{proposition}\label{Prop: full length charact of p-q-cox}
In a finite Coxeter group $W$ of rank $n$ with an element $g\in W$, the following are equivalent.
\begin{enumerate}
    \item The element $g\in W$ is a parabolic quasi-Coxeter element.
    \item The Hurwitz action on reduced reflection factorizations of $g$ is transitive.
    \item There exists a quasi-Coxeter element $w\in W$ such that $g\leq_T w$.
    \item The full reflection length of $g$ is given by $\displaystyle \ltr(g)=2n-\lR(g). $
\end{enumerate}
\end{proposition}
\begin{proof}
This is a combination of known results. The equivalence $(1)\sim (2)$ is \cite[Thm.~1.1]{BGRW}, while $(1)\sim (3)$ is given as \cite[Cor.~6.11]{BGRW}, and finally $(1)\sim (4)$ is \cite[Thm.~5.8]{DLM2}.
\end{proof}

We record here a corollary of the characterization that is implicit in the works \cite{BGRW} and \cite{gobet-cycle}, and of which we will also make use later on. 

\begin{corollary}\label{Cor: p-q-Cox heredit}
In a finite Coxeter group $W$ with a parabolic quasi-Coxeter element $g\in W$, if $W'\leq W$ is \emph{any} reflection subgroup that contains $g$, then $W_g\leq W'$ and $g$ is parabolic quasi-Coxeter in $W'$.
\end{corollary}
\begin{proof}
Let's start with a reduced reflection factorization $g=t_1\cdots t_k$ in $W'$. By Carter's lemma (Lemma~\ref{Lem: carter}), we have that $\lR[W'](g)=\codim(V^g)=\lR(g)$, so that this is also a reduced $W$-reflection factorization of $g$. By Proposition~\ref{prop:pqC factorization generates closure}, this forces the equality $\langle t_1,\ldots,t_k\rangle=W_g$, and therefore $W_g\leq W'$. 

Similarly, \emph{all} reduced $W'$-reflection factorizations of $g$ are also reduced in $W$.  Since $g$ is parabolic quasi-Coxeter in $W$, Proposition~\ref{Prop: full length charact of p-q-cox}~[$(1)\leftrightarrow (2)$] for the pair $(W,g)$ implies that they must all be Hurwitz equivalent. Applying Proposition~\ref{Prop: full length charact of p-q-cox}~[$(1)\leftrightarrow (2)$] again but now in the opposite direction and for the pair $(W',g)$ completes the proof.
\end{proof}

\section{Generating sets of reflections: minimum versus minimal}
\label{sec:min}

Let $W$ be a finite Coxeter group and $X$ a set of reflections that generates $W$.  We say that $X$ is a \defn{minimal generating set} if no proper subset of $X$ generates $W$.  For example, since every Coxeter group of rank $n$ may be generated by $n$ reflections but not fewer, every generating set of $n$ reflections is minimal.  However, there may be minimal generating sets of larger size.

\begin{example}[{\cite[Ch.~1, Exer.~6]{BjornerBrenti}}]
\label{ex:I2(30)}
Consider the dihedral group $I_2(30)$, of order $2 \times 30$, acting on $\RR^2$.  Let $R := \{r_0, r_2, r_{-3}\}$ be the set containing the following three reflections (illustrated in Figure~\ref{fig:I2(30)}): $r_0$ is the reflection across the $x$-axis, $r_2$ is the reflection across the line $y = \tan(2\pi/30) x$, and $r_{-3}$ is the reflection across the line $y = \tan(-3\pi/30) x$.

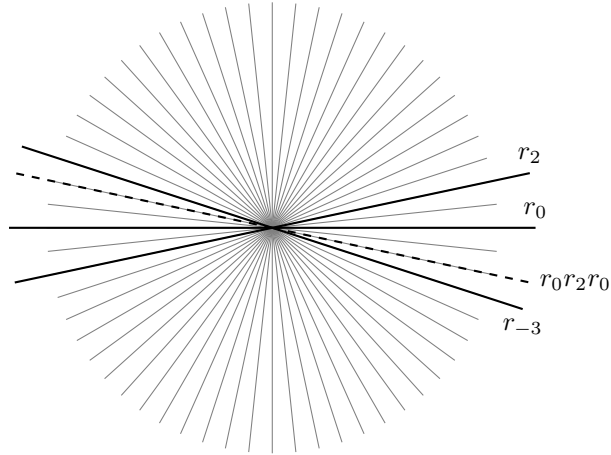
\begin{figure}
    \begin{tikzpicture}
        \foreach \x in {1,...,59}
            \draw[gray] (0,0) -- (6*\x:3cm);
        \draw[thick] (0,0) -- (0:3.5cm) node[above] {$r_0$};
        \draw[thick] (0,0) -- (0:-3.5cm);
        \draw[thick] (0,0) -- (12:3.5cm)node[above] {$r_2$};
        \draw[thick] (0,0) -- (12:-3.5cm);
        \draw[thick] (0,0) -- (-18:3.5cm)node[below] {$r_{-3}$};
        \draw[thick] (0,0) -- (-18:-3.5cm);
        \draw[thick, dashed] (0,0) -- (-12:3.5cm)node[right] {$r_0r_{2}r_0$};
        \draw[thick, dashed] (0,0) -- (-12:-3.5cm);
    \end{tikzpicture}
\caption{The reflecting lines associated with the group $I_2(30)$, as in Example~\ref{ex:I2(30)}.}
     \label{fig:I2(30)}
\end{figure}

It is easy to see that $R$ generates $I_2(30)$: for example, the pair of reflections $r_{-3}$ and $r_0r_{2}r_0$ have fixed lines that make an angle of $\pi/30$ and hence they generate $I_2(30)$.  It is also easy to see that the subset $\{r_0, r_2\}$ generates a subgroup of order $2 \times 15$, the subset $\{r_0, r_{-3}\}$ generates a subgroup of order $2 \times 10$, and the subset $\{r_2, r_{-3}\}$ generates a subgroup of order $2 \times 6$.  Thus $R$ is a minimal generating set of reflections, even though it is not (and does not contain) a \emph{mimumum-size} generating set of reflections.
\end{example}

This example motivates the following definition.

\begin{definition}\label{Defn: minimum=minimal}
Say that a finite Coxeter group $W$ of rank $n$ satisfies the \defn{minimum-equals-minimal property} if any set $X$ of reflections of $W$ that generates $W$ contains a subset $\{t_1,\ldots,t_n\}\subset X$ of exactly $n$ reflections that generates $W$.
Equivalently, the property is that every minim\emph{al} generating set of reflections is actually a generating set of minim\emph{um} size.  
\end{definition}

Since the reflections of $W_1 \times W_2$ are the union of the reflections of $W_1$ and the reflections of $W_2$, a finite Coxeter group has the minimum-equals-minimal property if and only if each of its irreducible factors also has the property.  Thus, it suffices to consider the case of irreducible groups.  In the next two results, we classify the irreducible finite Coxeter groups with the minimum-equals-minimal property, beginning with a generalization of Example~\ref{ex:I2(30)}.

\begin{proposition}\label{Prop: min-min dihedral case}
If $W$ is the dihedral group $I_2(m)$ and $m$ has at least three distinct prime factors, then $W$ does not satisfy the minimum-equals-minimal property.
\end{proposition}
\begin{proof}
As discussed in \cite[\S3.2]{LR},\footnote{However, the reader should note a small error in the discussion there: \cite{LR} omit the final ``$m$'' in the GCD, writing the condition as $\gcd(A_{12}, A_{13}, A_{23}) = 1$.  One can see the failure of this version already in the case $m = 3$, when three equally spaced vectors give $A_{12} = A_{13} = A_{23} = 2$ but the associated reflections generate the group.} for any three distinct reflections $r_1, r_2, r_3$ in $I_2(m)$, we can choose roots $\alpha_1, \alpha_2, \alpha_3$ orthogonal to their reflecting lines such that the angles between the $\alpha_i$ have measures $\frac{\pi}{m} A_{12}$, $\frac{\pi}{m} A_{13}$, and $\frac{\pi}{m} A_{23}$ where $A_{12}, A_{13}, A_{23}$ are integers in $\{1, \ldots, m - 1\}$ with sum $2m$; moreover, the three reflections generate the whole group if and only if $\gcd(A_{12}, A_{13}, A_{23}, m) = 1$.   Likewise, it's easy to see that reflections $r_i$ and $r_j$ generate the group if and only if $\gcd(A_{ij}, m) = 1$.  

Suppose that $m$ is divisible by at least three primes; then we can write $m = p \cdot q \cdot r$ where $p, q, r > 1$ are pairwise relatively prime (though not necessarily prime themselves).  By the Chinese Remainder Theorem, there are (unique) integers $a, b \in \{1, \ldots, m - 1\}$ that satisfy the systems of congruences
\[
\begin{cases}
a \equiv 0 \pmod p \\
a \equiv 1 \pmod q \\
a \equiv 1 \pmod r
\end{cases}
\textrm{ and }
\begin{cases}
b \equiv 1 \pmod p \\
b \equiv 0 \pmod q \\
b \equiv -1 \pmod r
\end{cases}.
\]
If $a + b > m$, define $A_{12} := a$, $A_{13} := b$, $A_{23} := 2m - a - b$; otherwise, define $A_{12} := m - a$, $A_{13} := m - b$, $A_{23} := a + b$. Now, pick an arbitrary root $\alpha_1$ of $I_2(m)$ and let $\alpha_2,\alpha_3$ be the uniquely determined roots so that the triple $(\alpha_1,\alpha_2,\alpha_3)$ corresponds to the triple $(A_{12},A_{13},A_{23})$, and let $r_1,r_2,r_3$ be the three associated reflections. In either case, by construction, $\gcd(A_{12}, m) = p$, $\gcd(A_{13}, m) = q$, and $\gcd(A_{23}, m) = r$, so no pair of $r_1, r_2, r_3$ generates the group.  However, the three reflections together do generate $I_2(m)$, as
\[
\gcd(A_{12}, A_{13}, A_{23}, m) = \gcd(a, b, m) = \gcd(\gcd(a, m), \gcd(b, m)) = \gcd(p, q) = 1.
\qedhere
\]
\end{proof}

\begin{lemma}\label{Lem: classification min=min}
Let $W$ be an irreducible finite Coxeter group. Then $W$ satisfies the minimum-equals-minimal property if and only if it belongs to one of the following two categories:
\begin{enumerate}
    \item it is not of dihedral type, or
    \item it is of dihedral type $I_2(m)$ with $m$ having at most two distinct prime factors.
\end{enumerate}
\end{lemma}
\begin{proof}
The ``only if'' direction is covered by Proposition~\ref{Prop: min-min dihedral case}. Thus, it remains to show that if a finite Coxeter group belongs to categories (1) or (2), then it has the minimum-equals-minimal property.  Our proof has two parts: \textbf{first}, we show that if a finite Coxeter group $W$ has a minimal but non-minimum generating set of reflections, then there is some (proper or not) reflection subgroup $W'\leq W$ with a minimal generating set of reflections of size exactly $\rk(W') + 1$; \textbf{second}, we show that in every irreducible group $W$ in the categories (1) and (2), every reflection generating set of size $\rk(W) + 1$ is non-minimal (i.e., contains a proper subset that also generates $W$). 

We now explain why these two claims suffice to prove the lemma.  If $W$ belongs to (1) or (2), then any reflection subgroup of $W$ also belongs to (1) or (2) (since each dihedral subgroup belongs to a parabolic dihedral subgroup, whose index $m$ must equal a label in the Coxeter diagram of $W$).  Therefore, supposing $W$ contains a minimal but non-minimum generating set of reflections, we have by the first claim that there is a reflection subgroup $W' \leq W$, belonging to (1) or (2), with a minimal generating set of $\rank(W') + 1$ reflections.  However, this is a direct contradiction with the second claim, and so $W$ cannot in fact have a minimal but non-minimum reflection generating set.  Therefore, it suffices to prove the two claims.

For the first claim, we proceed by induction on the cardinality of $W$. The base case $W=A_1$ is trivially true since $A_1$ does not have any minimal but non-minimum generating set. Now pick a finite Coxeter group $W$ with a minimal but non-minimum generating set $X$. If $|X| = n+1$ we are done.  If, on the other hand, $|X| \geq n+2$, choose a reflection $t \in X$, and let $W' = \langle X \smallsetminus t \rangle$. Since $X$ is a minimal generating set of $W$, $W' \lneq W$.  Moreover, the set $X \smallsetminus t$ must be a \emph{minimal} generating set for $W'$ (or else a proper subset of $X \smallsetminus t$ together with $t$ would generate $W$). Finally, since $|X \smallsetminus t|\geq n+1 > n\geq\rank(W')$ and we are done by the inductive hypothesis for $W'$.

For the second claim, we consider separately the infinite families $A_{n-1}=\Symm_n$, $B_n$, $D_n$, $I_2(m)$, and the exceptional types.

{\bf In the group $A_{n-1}=\Symm_n$}, any set $X$ of reflections corresponds to a graph on $[n]$ by identifying the transposition $(i j)$ with an edge joining vertices $i$ and $j$, and generating sets correspond to connected graphs (see, e.g., \cite[Prop.~2.1]{DLM1}).  Every connected graph contains a spanning tree.  The subset of $X$ corresponding to the spanning tree consists of $n-1$ reflections and generates $\Symm_n$. Therefore any generating set of reflections for $\Symm_n$ contains a generating subset of cardinality $n-1=\rank(\Symm_n)$.

{\bf In the group $B_n$}, any set of reflections corresponds to a signed graph on $[n]$ (with diagonal reflections corresponding to loops).  In order for a set of reflections to generate the whole group $B_n$, the graph must be connected (otherwise it generates a subgroup of a conjugate of $B_k \times B_{n - k}$ for some $k$) and must contain at least one loop (otherwise it generates a subgroup of $D_n$).  Any connected graph with a loop contains a spanning tree with a loop; the subset of reflections corresponding to this subgraph is a minimum generating set for $B_n$.

{\bf In the group $D_n$}, any set of reflections corresponds to a loopless signed graph on $[n]$.  In order for a set of reflections to generate the whole group $D_n$, the graph must be connected (otherwise it generates a subgroup of a conjugate of $D_k \times D_{n - k}$ for some $k$) and must contain at least one negative cycle (i.e., a cycle with an odd number of negative edges) (otherwise, it generates a subgroup conjugate to $\Symm_n$, since every signed spanning tree generates a conjugate of $\Symm_n$ \cite[Lem.~2.7]{Shi2005} that includes all the signed transpositions that do \emph{not} create a negative cycle with the tree).  Given a loopless connected signed graph with at least one negative cycle, fix a negative cycle, and delete edges from the graph one-by-one provided that the deleted edges do not belong to the cycle and removing them does not disconnect the graph.  The result is a signed \emph{unicycle} (a connected graph with exactly one cycle) whose unique cycle is the chosen negative cycle.  All unicycles on $n$ vertices have precisely $n$ edges, and all signed unicycles whose cycle is negative correspond to a generating set of reflections in $D_n$, so this subgraph corresponds to a minimum generating set for $D_n$.  

{\bf In the group $I_2(m)$ with $m$ having at most two (distinct) prime factors $p$ and $q$},
consider any three reflections $r_1,r_2,r_3$ that generate the whole group $I_2(m)$ and the associated integers $A_{12},A_{13},A_{23}$ described in the proof of Proposition~\ref{Prop: min-min dihedral case}. The integers $A_{ij}$ are constructed so that $A_{12}+A_{13}+A_{23}=2m$ and the assumption that the $r_i$'s generate the full group implies (again, see \cite[\S3.2]{LR}) that we must have $\gcd(A_{12}, A_{13}, A_{23}, m)=1$. From the first equation, if $p$ divides two of the $A_{ij}$ then it divides all three, which contradicts the second equation; likewise for $q$.  But then one of the $A_{ij}$ is divisible by neither $p$ nor $q$ and so is relatively prime to $m$; in this case the reflections $r_i$ and $r_j$ are sufficient to generate the whole group $W$.

{\bf In the exceptional types}, we used an exhaustive computer calculation, checking (for each group of rank $n$) all $(n + 1)$-element generating subsets of reflections and confirming that they each contain an $n$-element generating set.  (For the larger groups, this computation is made tractable by considering sets of reflections up to conjugation by $W$, as in \cite[\S3.6]{LR}.)
\end{proof}

\begin{remark}
For the simply laced types there is a simpler argument that combines two existing results: it was shown in \cite[Lem.\ 5.12]{BGRW} that if $W$ is simply laced, then a set of reflections in $W$ generates $W$ if and only if the roots orthogonal to their reflecting hyperplanes form a $\ZZ$-spanning set for the root lattice of $W$.  By \cite[Thm.\ 4.1]{BalnojanHertling}, if a set of roots generates the root lattice, then it contains a $\ZZ$-basis for the root lattice.  And finally again by \cite[Lem.\ 5.12]{BGRW}, the reflections corresponding to this $\ZZ$-basis form a minimum generating set for $W$.
\end{remark}

All reflection subgroups of Weyl groups are themselves Weyl groups and, moreover, the only irreducible dihedral groups that are crystallographic are $I_2(3)=\Symm_3$, $I_2(4)=B_2$, and $I_2(6)=G_2$. The following is then immediate.

\begin{corollary}
\label{cor:Weyl}
All Weyl groups $W$ satisfy the minimum-equals-minimal property of Definition~\ref{Defn: minimum=minimal}.
\end{corollary}

\begin{remark}
For real reflection groups, Proposition~\ref{Prop: min-min dihedral case} establishes that \emph{minimal-but-not-minimum} generating sets of reflections are reasonably well behaved: restricted to dihedral types and controlled by some elementary number theory.  In complex reflection groups, generating sets with this property are much more complicated; perhaps this is a reason why Conjecture~\ref{lewis conjecture} is truly a conjecture for the real types and fails in some complex cases.
\end{remark}

\section{Main result}
\label{sec: main result}

In this section, we prove our main result on the Hurwitz orbits of factorizations of parabolic quasi-Coxeter elements.  We begin with the case of minimum-length full reflection factorizations.

\begin{proposition}\label{Prop: hurw trans on min full}
If $W$ is a finite Coxeter group and $g\in W$ a parabolic quasi-Coxeter element, then the Hurwitz action is transitive on minimum-length full reflection factorizations of $g$.
\end{proposition}
\begin{proof}
Denote the reflection length of $g$ by $k$ (i.e., $k=\lR(g)$).  By Proposition~\ref{Prop: full length charact of p-q-cox}~[$(1)\leftrightarrow (4)$], the minimum length of full factorizations of $g$ is $2n-k$, where $n$ is the rank of $W$.  By Lemma~\ref{lem: LR lemma}, every reflection factorization of $g$ as given in the statement is Hurwitz-equivalent to one of the form 
\begin{equation}
g= t_1\cdots  t_k\cdot  t_{k+1} \cdot  t_{k+1}\cdots  t_n \cdot  t_n,
\label{EQ: g-facn-1}
\end{equation}
where $ t_1\cdots  t_k=g$ is a reduced reflection factorization of $g$. It is sufficient to show that the factorization \eqref{EQ: g-facn-1} is Hurwitz-equivalent to any other of the form 
\begin{equation}
g= t'_1\cdots  t'_k\cdot  t'_{k+1} \cdot  t'_{k+1}\cdots  t'_n  \cdot t'_n
\label{EQ: g-facn-2}
\end{equation}
with $\langle t'_i\rangle=W$. 
We proceed by induction on the common length $2n - k$ of the two factorizations.  In the base case, the length takes the smallest possible value $2n-k=n$.  This means that the factorization $g = t_1 \cdots t_n$ is simultaneously reduced and full, so $g$ is a quasi-Coxeter element and Proposition~\ref{Prop: full length charact of p-q-cox}~[$(1)\leftrightarrow (2)$] applies. 

Now assume that $2n - k > n$.  We aim to find a factorization in the same Hurwitz orbit as \eqref{EQ: g-facn-1} whose last factor is $t'_n$.  We first turn \eqref{EQ: g-facn-1} to a factorization of the form 
\begin{equation}
g= t_1\cdots  t_n\cdot  t''_{n+1}\cdots  t''_{2n-k},\label{EQ: g-facn-3}
\end{equation}
by sliding some terms $ t_i$ with left Hurwitz moves \eqref{eq:Hurwitz move}.  Since the factorization \eqref{EQ: g-facn-1} is full, we must further have that $\langle  t_i\rangle=W$, and then by Proposition~\ref{Prop: t_1...t_n=w is q-Cox} that $w:= t_1\cdots  t_n$ is a quasi-Coxeter element of $W$.  Since $w= t_1\cdots  t_n$ is quasi-Coxeter, its parabolic closure is the full group $W$ and hence (by Lemma~\ref{Lem: carter}) all reflections lie below $w$ in absolute order $\leq_{\RRR}$. In particular, $w$ has a reduced reflection factorization whose last factor is $ t'_n$.   Since $w$ is quasi-Coxeter, we can reach that factorization from $t_1 \cdots t_n$ via Hurwitz moves; then we can slide $ t'_n$ to the last position of \eqref{EQ: g-facn-3} via right Hurwitz moves \eqref{eq:inverse Hurwitz move}. This all means that the first factorization \eqref{EQ: g-facn-1} of $g$ is Hurwitz-equivalent to one of the following form: 
\begin{equation}
g=\widehat{ t}_1\cdots \widehat{ t}_{2n-k-1}\cdot  t_n'.
\label{EQ: g-facn-4}
\end{equation}
We would now like to show that the factorizations \eqref{EQ: g-facn-2} and \eqref{EQ: g-facn-4} are Hurwitz-equivalent. It is sufficient to prove the equivalence of the following two factorizations:
\begin{equation}
g t'_n=\widehat{ t}_1\cdots \widehat{ t}_{2n-k-1}\qquad\text{and}\qquad gt_n'= t'_1\cdots  t'_k\cdot  t'_{k+1} \cdot t'_{k+1}\cdots  t'_{n-1}\cdot t'_{n-1}\cdot  t_n'.
\label{Eq: two facns}
\end{equation}
Define 
\begin{equation}
    w' := t'_1\cdots t'_k\cdot t'_{k+1}\cdots t'_n = t'_1 \cdots t'_k \cdot t'_n \cdot (t'_n t'_{k + 1} t'_n) \cdots (t'_n t'_{n - 1} t'_n).
    \label{Eq: w'=reduced}
\end{equation}
The element $w'$ is quasi-Coxeter for the same reasons that $w$ is, and the two factorizations in \eqref{Eq: w'=reduced} are reduced factorizations of $w'$.  Therefore, from the second factorization in \eqref{Eq: w'=reduced}, we have that $t'_1 \cdots t'_k t'_n = gt'_n$ lies below $w'$ in the absolute order $\leq_{\RRR}$.  Then by Proposition~\ref{Prop: full length charact of p-q-cox}~[$(1)\leftrightarrow (3)$], $gt'_n$ is a parabolic quasi-Coxeter element.  It follows that the two factorizations in \eqref{Eq: two facns} are length-$(2n - k - 1)$ reflection factorizations of a parabolic quasi-Coxeter element whose reflection length is $k + 1$;  the second is manifestly full (it shares the same set of factors as the full factorization \eqref{EQ: g-facn-2}), and so to conclude using the inductive hypothesis it suffices to show that the first factorization in \eqref{Eq: two facns} is also full.

Let $W':=\langle \widehat{t}_1,\ldots,\widehat{t}_{2n-k-1}\rangle$.  By construction we have that $\langle W', t'_n\rangle =W$. Since $gt'_n$ is parabolic quasi-Coxeter in $W$ and also belongs to $W'$, we have by Corollary~\ref{Cor: p-q-Cox heredit} that $W'$ contains the parabolic closure $W_{gt'_n}=\langle t'_1,\ldots,t'_k,t'_n\rangle$. In particular, $W'$ contains $t'_n$.  Since $\langle W', t'_n\rangle = W$, this means $W' = W$; that is, the first factorization of \eqref{Eq: two facns} is indeed full. Now the inductive assumption guarantees that the two factorizations in \eqref{Eq: two facns} are Hurwitz-equivalent, and so (by the preceding arguments) that the factorizations \eqref{EQ: g-facn-1} and \eqref{EQ: g-facn-2} are equivalent.  The result follows by induction.
\end{proof}

We are now prepared for the proof of our main result.

\begin{theorem}\label{thm: main}
Conjecture~\ref{lewis conjecture} is true whenever $g$ is a parabolic quasi-Coxeter element in the finite Coxeter group $W$.
\end{theorem}
\begin{proof}
It is sufficient to prove the theorem for irreducible groups.  The case of the dihedral groups has already been covered in \cite{berger}.  By Lemma~\ref{Lem: classification min=min}, all the non-dihedral types satisfy the minimum-equals-minimal property of Definition~\ref{Defn: minimum=minimal}, and so it is enough to prove it for this class of groups.  Therefore, let $W$ be a finite Coxeter group with the minimum-equals-minimal property and let $g$ be a parabolic quasi-Coxeter element of $W$, and consider two reflection factorizations of $g$ that generate the same subgroup $W' \leq W$ and have the same multiset of $W'$-conjugacy classes.

By Corollary~\ref{Cor: p-q-Cox heredit}, $g$ is parabolic quasi-Coxeter in $W'$.  Further, it follows from the proof of Lemma~\ref{Lem: classification min=min} that if $W$ has the minimum-equals-minimal property, then the same is true of all of its reflection subgroups; in particular, it is true of $W'$.  So without loss of generality we may as well relabel $W'$ as $W$ and consider the case that the two factorizations of $g$ are full.  Let $k = \lR(g)$ be the reflection length of $g$ and let $n = \rank(W)$ be the rank of $W$, so that the length of the two factorizations is at least $\ltr(g)=2n- k$ (by Proposition~\ref{Prop: full length charact of p-q-cox} [$(1) \leftrightarrow (4)$]).  We now proceed by induction on the length of the factorizations.  

The base case is Proposition~\ref{Prop: hurw trans on min full}.  Now suppose the two full factorizations of $g$ have length $2n -k + 2s$ for $s > 0$.  By Lemma~\ref{lem: LR lemma}, we can assume that the two factorizations have the form 
\begin{equation}
\label{eq:two factorizations}
g=t_1\cdots t_k\cdot t_{k+1} \cdot t_{k+1} \cdots t_{n+s} \cdot t_{n + s}\qquad \text{and}\qquad g=t'_1\cdots t'_k\cdot t'_{k+1}\cdot t'_{k+1}\cdots t'_{n+s}\cdot t'_{n+s}.
\end{equation}
Our strategy is to apply some Hurwitz moves to make the two factorizations agree in their last two factors while remaining full factorizations of $g$.  Toward that end, observe that $\{t_1, \ldots, t_{n + s}\}$ and $\{t'_1, \ldots, t'_{n + s}\}$ are generating sets of $W$ (because both factorizations are full) that are not minimum (because $s > 0$).  Therefore, since $W$ is assumed to have the minimum-equals-minimal property, in each set there is a reflection that can be removed to leave a generating set of reflections.  Moreover, since the two factorizations have the same multiset of conjugacy classes and (by Lemma~\ref{Cor. refl gen sets same conj classes}) all minimal generating sets have the same multiset of conjugacy classes, we can even arrange to choose the ``unnecessary'' reflections $t_i$ and $t'_j$ to be conjugate in $W$.  We now explain how to produce factorizations in the same Hurwitz orbits as those in \eqref{eq:two factorizations} that end in a pair of equal factors $\cdots t' \cdot t'$ that are conjugate to $t_i$.

Suppose first that $i > k$.  In that case, we can use Hurwitz moves to slide the two copies of $t_i$ to the end of the factorization, producing a new factorization
\[
g = t_1 \cdots t_k \cdot t_{k + 1} \cdot t_{k + 1} \cdots t_{i - 1} \cdot t_{i - 1} \cdot t_{i + 1} \cdot t_{i + 1} \cdots t_{n + s} \cdot t_{n + s} \cdot t_i \cdot t_i
\]
in which the prefix is a full reflection factorization of $g$.

On the other hand, suppose that $i \leq k$.  Apply the Hurwitz moves $\sigma_{i - 1}^{-1}$, $\sigma_{i - 2}^{-1}$, \ldots, $\sigma_1^{-1}$ in that order to produce a new factorization $g = t' \cdot t_1 \cdots t_{i - 1} \cdot t_{i + 1} \cdots t_k \cdot t_{k + 1} \cdot t_{k + 1} \cdots t_{n + s} \cdot t_{n + s}$ for some reflection $t'$ that belongs to the same conjugacy class as $t_i$ and still lies below $g$ is absolute order.  Since the remaining factors have not changed, they still include an $n$-element generating set for $W$.  Now apply the Hurwitz moves $\sigma_1^{-1}$, $\sigma_2^{-1}$, \ldots, $\sigma_{2n - k + 2s - 1}^{-1}$ in that order to produce a new factorization in which $t'$ is at the end (instead of the beginning) and all other factors have been conjugated by $t'$.  Conjugating a generating set gives another generating set, so the remaining factors (omitting $t'$) contain an $n$-element generating set; we can use Hurwitz moves to bring these $n$ factors to consecutive positions.  As in the proof of Proposition~\ref{Prop: hurw trans on min full}, the product $w$ of these $n$ factors is a quasi-Coxeter element (Proposition~\ref{Prop: t_1...t_n=w is q-Cox}) and $t' \leq_{\RRR} w$ (Lemma~\ref{Lem: carter}), so by Proposition~\ref{Prop: hurw trans on min full} we can apply some Hurwitz moves to these $n$ factors and produce a factorization with a second copy of $t'$.  We may again use Hurwitz moves to slide this second copy of $t'$ to the end, and then apply Lemma~\ref{lem: LR lemma} to the prefix of $2n - k + 2s - 2$ factors to produce a new factorization of the form
\begin{equation}
    g = t''_1 \cdots t''_k \cdot t''_{k + 1} \cdot t''_{k + 1} \cdots t''_{n + s - 1} \cdot t''_{n + s - 1} \cdot t' \cdot t'.
\label{so many factorizations}
\end{equation}
This is in the desired form, but we must explain why the prefix (omitting the final two $t'$ factors) is full; this will finally use the fact that $t' \leq_{\RRR} g$.  Let $W' = \langle t''_1, \ldots, t''_{n + s - 1}\rangle$.  By Corollary~\ref{Cor: p-q-Cox heredit}, $W_g \leq W'$, and by Lemma~\ref{Lem: carter}, $t' \in W_g$.  Therefore $t' \in W'$.  Thus $W' = \langle t''_1, \ldots, t''_{n + s - 1}, t'\rangle = W$ (where the last equality comes because the group generated by the factors is invariant under the Hurwitz action).  Thus the prefix (dropping the two copies of $t'$) is full, as claimed.

The last two paragraphs show that the first factorization in \eqref{eq:two factorizations} is Hurwitz-equivalent to a factorization in which the last two factors are equal and come from a prescribed conjugacy class, and for which the prefix (omitting those factors) is a full factorization of $g$.  The same argument applies to the second factorization in \eqref{eq:two factorizations}; and, moreover, because the two factorizations in \eqref{eq:two factorizations} have the same multiset of conjugacy classes, we can do this in such a way that the final two factors in the two new factorizations belong to the same conjugacy class.  To finish, we use the same technique as in \cite[Thm.~1.1]{LR}, using Hurwitz moves to bring factors from the prefix to the third position from the end, applying the Hurwitz moves (with $M=2n-k+2s$)
\[
(\ldots, t, t', t') \xrightarrow{\sigma_{M-2}} (\ldots, t', t'tt', t') \xrightarrow{\sigma_{M-1}} (\ldots, t', t', t)\xrightarrow{\sigma_{M-1}} (\ldots, t', t, tt't) \xrightarrow{\sigma_{M-2}} (\ldots, t, tt't, tt't),
\]
and then restoring the prefix to its original form.  Since the prefix generates $W$, we can by repeated application of this strategy replace the final two factors with any conjugate pair of reflections; in particular, we may arrange so that the two factorizations in question agree on their final two factors.  Since the prefixes are full and have shorter length, we are done by induction.
\end{proof}

\bibliographystyle{alpha}
\bibliography{biblio}

\end{document}